\documentclass[11pt]{amsart}

\usepackage[T1]{fontenc}
\usepackage[utf8]{inputenc}
\usepackage{lmodern}
\usepackage{microtype}
\usepackage[margin=1.22in]{geometry}
\usepackage{amsmath,amssymb,amsfonts,amsthm,mathtools}
\usepackage{mathrsfs}
\usepackage{enumitem}
\usepackage[colorlinks=true,linkcolor=blue,citecolor=blue,urlcolor=blue]{hyperref}
\usepackage[nameinlink,noabbrev]{cleveref}

\numberwithin{equation}{section}

\theoremstyle{plain}
\newtheorem{theorem}{Theorem}[section]
\newtheorem{proposition}[theorem]{Proposition}

\newtheorem{lemma}[theorem]{Lemma}
\newtheorem{conjecture}[theorem]{Conjecture}

\theoremstyle{definition}

\newtheorem{problem}[theorem]{Problem}

\theoremstyle{remark}

\newcommand{\C}{\mathbb C}

\newcommand{\Sym}{\mathfrak S}
\newcommand{\sgn}{\operatorname{sgn}}
\newcommand{\Fix}{\operatorname{Fix}}

\newcommand{\per}{\operatorname{per}}
\newcommand{\imm}{\operatorname{imm}}

\newcommand{\la}{\lambda}

\title[On immanants of Cayley tables]{On immanants of Cayley tables}
\author{Xuan Wang, Hanbin Zhang}
	\address{School of Mathematics (Zhuhai), Sun Yat-sen University, Zhuhai 519082, Guangdong, P.R. China}
	\email{wangx728@mail2.sysu.edu.cn, zhanghb68@mail.sysu.edu.cn}

\begin{document}

\begin{abstract}
Let $G$ be a finite abelian group of order $n$ and let \(
        \mathcal M_G=(x_{a+b})_{a,b\in G}
\)
be the Cayley table of $G$.  Let $\imm_\lambda(\mathcal M_G)$ be the immanant of $\mathcal M_G$ with respect to a partition $\lambda$ and $\mathcal I_\lambda(G)$ be the number of formally different monomials occurring in $\imm_\lambda(\mathcal M_G)$ (in particular, we denote by $\mathcal P(G)$ (resp. $\mathcal D(G)$) for the corresponding quantity for $\per(\mathcal M_G)$ (resp. $\det(\mathcal M_G)$) for simplicity). The study of $\mathcal P(G)$ and $\mathcal D(G)$ lies at the intersection of algebraic combinatorics and additive combinatorics. In this paper, we prove the following results.

\begin{enumerate}
    \item If $|G|$ is a prime power, then
    $$\mathcal P(G)=\mathcal D(G).$$
    \item If $|G|$ is odd, then
\[
        \mathcal I_{(n-1,1)}(G)= \mathcal I_{(2,1^{n-2})}(G)=0,
\]
and that if $|G|\equiv 2\pmod 4$, then
\[
        \mathcal I_{(n-1,1)}(G)=\mathcal P(G)\quad
        \text{and}\quad
        \mathcal I_{(2,1^{n-2})}(G)=\mathcal D(G).
\] 
    \item If $|G|$ is odd and $|G|\ge 7$, then
\[
        \imm_{(4,1^{n-4})}(\mathcal M_G)=\imm_{(2,2,2,1^{n-6})}(\mathcal M_G).
\]
\end{enumerate}

\end{abstract}


\maketitle

\section{Introduction}

For an $n\times n$ matrix $\mathcal M=(m_{i,j})_{1\le i,j\le n}$, its immanant with respect to a partition $\lambda$ is defined as
$$\mathsf {imm}_{\lambda}(\mathcal M)=\sum_{\tau\in \Sym_n}\chi^{\lambda}(\tau)\prod_{i=1}^{n}m_{i,\tau(i)}.$$
Note that, determinant (resp. permanent) is the immanant corresponding to the sign character (resp. the trivial character):
\[
   \imm_{(1^n)}(\mathcal M)=\det(\mathcal M),
   \qquad
   \imm_{(n)}(\mathcal M)=\per (\mathcal M).
\]
The prehistory of immanants goes back to Schur's work on Hermitian forms and characters, he considered generalized matrix functions associated with group characters and proved fundamental inequalities for positive semidefinite Hermitian matrices; see \cite{Schur1918,MarcusMinc1965}.  The word and systematic study of immanants entered the literature through Littlewood and Richardson's work \cite{LittlewoodRichardson1934a,LittlewoodRichardson1934b}, their papers placed immanants in the representation theory of symmetric groups and in the emerging language of symmetric functions. Classical works on immanants developed along several directions: matrix inequalities and the permanental-dominance program \cite{Marcus1963,Lieb1966,MerrisWatkins1985,Pate1994,Stembridge1992}; combinatorial and symmetric-function formulas for Jacobi--Trudi and related matrices \cite{GouldenJackson1992,GouldenJackson1992b,Greene1992,StanleyStembridge1993}; and positivity phenomena on totally positive or totally nonnegative matrices \cite{Stembridge1991,RhoadesSkandera2006,Pylyavskyy2010,LY}.

Subsequent work connected immanants with Hecke algebras, Kazhdan--Lusztig theory, Temperley--Lieb quotients, and quantum matrix identities \cite{Haiman1993,RhoadesSkandera2005,RhoadesSkandera2006,RhoadesSkandera2010,KonvalinkaSkandera2011}; with probabilistic models such as finite point processes \cite{DiaconisEvans2000}; and with computational complexity, where determinant and permanent sit at opposite ends of a sharp complexity landscape for immanant families \cite{Hartmann1985,Curticapean2021}.  Recent contributions continue to enlarge this circle of ideas, for instance through new hook-shape character formulas and quasisymmetric refinements \cite{Lesnevich2024,Campbell2025}.  In short, immanants present a rich interplay between representation theory, algebraic combinatorics, matrix theory, probability, and complexity.

In this paper, we study a highly structured specialization of the immanant problem, namely the Cayley table of a finite abelian group.  Let $G$ be a finite abelian group of order $n$, written additively, and put
\[
        \mathcal M_G=(x_{a+b})_{a,b\in G},
\]
where $x_g$ $(g\in G)$ are commuting indeterminates.  For a partition $\lambda\vdash n$, let $\chi^\lambda$ denote the irreducible character of $\Sym_n$ indexed by $\lambda$.  
After identifying the rows and columns with $G$, we may write
\[
        \imm_\lambda(\mathcal M_G)=\sum_{\sigma\in \Sym(G)} \chi^\lambda(\sigma)
        \prod_{a\in G}x_{a+\sigma(a)}.
\]
All characters are understood through the natural identification of $\Sym(G)$ with $\Sym_n$; character values therefore depend only on cycle type.

The study of $\det(\mathcal M_G)$ can be traced back to seminal works of Dedekind and Frobenius, which led to the foundation of representation theory of groups; see \cite{John}.

In 1952, with an elegant constructive approach, Hall \cite{Hall1952} studied $\per(\mathcal M_G)$ and proved the following interesting result.

\begin{theorem}{\rm(\cite{Hall1952})}\label{Hal}
    Let $G=\{g_0,\cdots,g_{n-1}\}$ be an abelian group of order $n$ with $g_0$ the identity. Then the monomial $x_{g_0}^{\lambda_0}x_{g_1}^{\lambda_1}\cdots x_{g_{n-1}}^{\lambda_{n-1}}$ occurs in $\per(\mathcal M_{G})$ if and only if $$\lambda_0+\cdots +\lambda_{n-1}=n\quad \text{and}\quad
{\lambda_0}g_0+{\lambda_1}g_1+\cdots+{\lambda_{n-1}} g_{n-1}=g_0,$$
    where, for any $0\leq j\leq n-1$, $\lambda_jg_j \coloneqq g_j+\cdots+ g_j$ ($\lambda_j$ times).
\end{theorem}

Note that, a sequence $S=g_0^{\lambda_0}g_1^{\lambda_1}\cdots g_{n-1}^{\lambda_{n-1}}$ satisfying ${\lambda_0}g_0+{\lambda_1}g_1+\cdots+{\lambda_{n-1}} g_{n-1}=g_0$ is called a zero-sum sequence over $G$ of length $n$, which is an important object in additive combinatorics.
Let $\mathcal I_\lambda(G)$ be the number of formally different monomials occurring in $\imm_\lambda(\mathcal M_G)$. In particular, we denote by $\mathcal P(G)$ (resp. $\mathcal D(G)$) for the corresponding quantity for $\per(\mathcal M_G)$ (resp. $\det(\mathcal M_G)$) for simplicity. For example, for
    $$\mathcal M_{C_3}=\begin{pmatrix}
    x_0 & x_1 & x_2\\
    x_1 & x_2 & x_0\\
    x_2 & x_0 & x_1
    \end{pmatrix},
    $$
we have $\mathsf{det}(\mathcal M_{C_3})=-x_0^3-x_1^3-x_2^3+3x_0x_1x_2$ and $\mathsf {per}(\mathcal M_{C_3})=x_0^3+x_1^3+x_2^3+3x_0x_1x_2$, which implies that $\mathcal P(C_3)=\mathcal D(C_3)=4$. Due to possible cancellation, we have $\mathcal D(G)\le \mathcal P(G)$. Li and Zhang \cite{LiZhang2024} proved that, for finite abelian groups $G$ and $H$, $\mathcal P(G)=\mathcal P(H)$ if and only if $G\cong H$. Using symmetric functions, Thomas \cite{Thomas2004} proved that, if $n$ is a prime power, $\mathcal D(C_n)= \mathcal P(C_n)$. Later, Colarte, Mezzetti, Mir\'o-Roig, and Salat proved the converse statement, i.e., $\mathcal D(C_n)= \mathcal P(C_n)$ only if $n$ is a prime power; see \cite{Colarte2019}. Later Panyushev \cite{Panyushev2011} proposed the following problem.
\begin{problem}
     What are necessary/sufficient conditions on a finite abelian group $G$ for the equality $\mathcal D(G)= \mathcal P(G)$? Specifically, is it still true that the condition ‘$|G|$ is a
prime power’ is sufficient?
\end{problem}
Our first main result is the following.

\begin{theorem}\label{PGDG}
    Let $G$ be a finite abelian group and $|G|$ is a prime power, then we have
    $$\mathcal P(G)=\mathcal D(G).$$
\end{theorem}

Recently, Wang, Zhang, and Zhang \cite{WangZhangZhang2026} initiated a systematic study of general immanants of $\mathcal M_G$ and formulated several conjectures for cyclic groups:
\begin{conjecture}\label{mainconj}
\begin{enumerate}
    \item For any $n \ge 3$, if $n$ is odd, then $\mathcal I_{(n-1,1)}(C_n)= \mathcal I_{(2,1^{n-2})}(C_n)=0$.
    \item For any $n \ge 2$, if $n\equiv 2\pmod 4$, then $\mathcal I_{(n-1,1)}(C_n)=\mathcal P(C_n)$ and $\mathcal I_{(2,1^{n-2})}(C_n)=\mathcal D(C_n)$.
    \item For any $n \ge 7$, if $n$ is odd, then $\mathcal I_{(n-2,1,1)}(C_n)= \mathcal P(C_n)$.
\end{enumerate}
\end{conjecture}

Our second main result proves (1) and (2) of Conjecture \ref{mainconj} in a generalized version.

\begin{theorem}\label{near-hookimm}
    Let $G$ be a finite abelian group.
\begin{enumerate}
    \item If $|G|$ is odd, then
\[
        \mathcal I_{(n-1,1)}(G)= \mathcal I_{(2,1^{n-2})}(G)=0.
\]
    \item If $|G|\equiv 2\pmod 4$, then
\[
        \mathcal I_{(n-1,1)}(G)=\mathcal P(G)\quad
        \text{and}\quad
        \mathcal I_{(2,1^{n-2})}(G)=\mathcal D(G).
\]
\end{enumerate}

\end{theorem}

Based on some numerical experiments via SageMath, we observed several interesting phenomena concerning $\imm_{\lambda}(\mathcal M_G)$ for general $\lambda$ and we shall prove the following neat result, which shows that two seemingly unrelated immanants (with respect to non-conjugate partitions) are identical.

\begin{theorem}\label{immtwins}
    Let $G$ be a finite abelian group with $|G|$ odd and $|G|\ge 7$, then
\[
        \imm_{(4,1^{n-4})}(\mathcal M_G)=\imm_{(2,2,2,1^{n-6})}(\mathcal M_G).
\]
\end{theorem}

The paper is organized as follows. In Section~\ref{sec:detper}, we study $\mathcal P(G)$ and $\mathcal D(G)$. In Section~\ref{sec:hall}, we consider $\imm_{(n-1,1)}(\mathcal M_G)$ and $\imm_{(2,1^{n-2})}(\mathcal M_G)$. In Section~\ref{sec:twin}, we prove Theorem \ref{immtwins}.  We end this paper with some concluding remarks.

\section{Determinant and permanent supports}\label{sec:detper}

In this section, we consider the relation between $\mathcal D(G)$ and $\mathcal P(G)$, and provide the proof of Theorem \ref{PGDG}.

It is convenient to use the Toeplitz group matrix
\[
        A_G=(x_{a-b})_{a,b\in G}.
\]
The matrices $A_G$ and $\mathcal M_G$ differ by the column permutation $b\mapsto -b$, so they have the same determinant and permanent supports.

Let $S=g_1\cdots g_n$ be a sequence over $G$ of length $n$.  If $\alpha_g$ is the multiplicity of $g$ in $S$, define the labelled determinant coefficient
\[
        \widetilde c_G(S)=\left(\prod_{g\in G}\alpha_g!\right)[x_{g_1}\cdots x_{g_n}]\det(A_G).
\]
Let $\Pi_0(S)$ be the set of set partitions $\pi$ of $\{1,\ldots,n\}$ such that every block is zero-sum:
\[
        \sum_{i\in B}g_i=0\qquad(B\in\pi).
\]

\begin{theorem}\label{thm:partition-coeff}
For every sequence $S$ over $G$ of length $n$,
\begin{equation}\label{eq:partition-coeff}
        \widetilde c_G(S)=\sum_{\pi\in\Pi_0(S)}(-1)^{n-|\pi|}n^{|\pi|}
        \prod_{B\in\pi}(|B|-1)!.
\end{equation}
\end{theorem}

\begin{proof}
The abelian group determinant factorization gives
\[
        \det(A_G)=\prod_{\chi\in\widehat G}\left(\sum_{g\in G}\chi(g)x_g\right).
\]
Therefore
\[
        \widetilde c_G(S)=\sum_{\theta:\{1,
        \ldots,n\}\to \widehat G\,\text{bijection}}
        \prod_{i=1}^n\theta(i)(g_i).
\]
Impose bijectivity by M\"obius inversion on the partition lattice $\Pi_n$.  The M\"obius function from the discrete partition $\hat 0$ (the minimal element) to $\pi$ is
\[
        \mu(\hat0,\pi)=(-1)^{n-|\pi|}\prod_{B\in\pi}(|B|-1)!.
\]
Thus
\[
\widetilde c_G(S)=\sum_{\pi\in\Pi_n}\mu(\hat0,\pi)
        \prod_{B\in\pi}\left(\sum_{\chi\in\widehat G}\chi\left(\sum_{i\in B}g_i\right)\right).
\]
By the orthogonality of characters, the inner sum is $n$ if the block sum is zero and $0$ otherwise.  Hence precisely the partitions in $\Pi_0(S)$ survive, giving \eqref{eq:partition-coeff}.
\end{proof}

\begin{proof}[Proof of Theorem \ref{PGDG}.]
Let $n=|G|=p^r$ and let $S=(g_1,\ldots,g_n)$ be a zero-sum sequence.  We show $\widetilde c_G(S)\ne0$.  In \eqref{eq:partition-coeff}, the one-block partition belongs to $\Pi_0(S)$ and contributes
\[
        T_1=(-1)^{n-1}n(n-1)!.
\]
Its $p$-adic valuation is $r+v_p((n-1)!)$.

Consider another zero-sum partition with $k\ge2$ blocks of sizes $b_1,
\ldots,b_k$.  Its term has valuation
\[
        kr+\sum_{j=1}^k v_p((b_j-1)!).
\]
It suffices to prove
\begin{equation}\label{eq:padic-ineq}
        v_p\left(\frac{(n-1)!}{\prod_{j=1}^k(b_j-1)!}\right)<(k-1)r.
\end{equation}
Put $m_j=b_j-1$, so $\sum_jm_j=n-k$.  Then
\[
\frac{(n-1)!}{\prod_j(b_j-1)!}
=(n-1)(n-2)\cdots(n-k+1)\binom{n-k}{m_1,\ldots,m_k}.
\]
Since $n=p^r$ and $1\le i\le k-1<n$, we have $v_p(n-i)=v_p(i)$, so the first factor has valuation $v_p((k-1)!)$.  Also $n-k<n=p^r$, and repeated use of $v_p\binom{M}{d}<r$ for $M<p^r$ gives
\[
        v_p\binom{n-k}{m_1,\ldots,m_k}\le(k-1)(r-1).
\]
Therefore the valuation in \eqref{eq:padic-ineq} is at most
\[
        v_p((k-1)!)+(k-1)(r-1)<(k-1)r.
\]
Thus the one-block term is the unique term of minimal $p$-adic valuation in \eqref{eq:partition-coeff}.  The sum cannot vanish.
\end{proof}

\section{On $\imm_{(n-1,1)}(\mathcal M_G)$ and $\imm_{(2,1^{n-2})}(\mathcal M_G)$}\label{sec:hall}

In this section, we consider the immanants $\imm_{(n-1,1)}(\mathcal M_G)$ and $\imm_{(2,1^{n-2})}(\mathcal M_G)$, and then prove Theorem \ref{near-hookimm}.

For a monomial $m=\prod_{a\in G}x_a^{\lambda_a}$ of degree $n$, define
\[
\mathcal{P}(m):=\Bigl\{\sigma\in \Sym(G):\prod_{u\in G}x_{u+\sigma(u)}=m\Bigr\}.
\]
Then $p_m=|\mathcal{P}(m)|$, while
\[
d_m=\sum_{\sigma\in \mathcal{P}(m)}\sgn(\sigma).
\]
For $a\in G$, set
\[
\mathcal{P}_a(m):=\{\sigma\in \mathcal{P}(m):\sigma(0)=a\}.
\]

The proof of the following lemma employ ideas from \cite{WangZhangZhang2026}.

\begin{lemma}\label{lem:translation-action}
For each $\gamma\in G$, define a permutation action on $\Sym(G)$ by
\[
(\gamma\star \sigma)(u)=\sigma(u-\gamma)-\gamma.
\]
Then $\sgn(\gamma\star\sigma)=\sgn(\sigma)$ and
\[
\prod_{u\in G}x_{u+(\gamma\star\sigma)(u)}
=
\prod_{u\in G}x_{u+\sigma(u)}.
\]
In particular, $\gamma\star\sigma\in \mathcal{P}(m)$ whenever $\sigma\in \mathcal{P}(m)$.
\end{lemma}

\begin{proof}
Let $\tau_{\gamma}(u)=u+\gamma$ be translation by $\gamma$. Then
\[
\gamma\star\sigma=\tau_{-\gamma}\circ \sigma\circ \tau_{-\gamma},
\]
so
\[
\sgn(\gamma\star\sigma)=\sgn(\tau_{-\gamma})^2\sgn(\sigma)=\sgn(\sigma).
\]
Also,
\[
\prod_{u\in G}x_{u+(\gamma\star\sigma)(u)}
=
\prod_{u\in G}x_{u+\sigma(u-\gamma)-\gamma}
=
\prod_{v\in G}x_{v+\sigma(v)},
\]
after the change of variables $v=u-\gamma$.
\end{proof}

\begin{lemma}\label{lem:Pa-count}
For every $a\in G$,
\[
|\mathcal{P}_a(m)|=\frac{\la_a}{n}\,p_m,
\]
and
\[
\sum_{\sigma\in \mathcal{P}_a(m)}\sgn(\sigma)=\frac{\la_a}{n}\,d_m.
\]
\end{lemma}

\begin{proof}
Consider the set
\[
X_a(m):=\{(\sigma,u)\in \mathcal{P}(m)\times G: u+\sigma(u)=a\}.
\]
If $\sigma\in \mathcal{P}(m)$, then the symbol $x_a$ occurs exactly $\la_a$ times in the product $\prod_{u\in G}x_{u+\sigma(u)}$, so there are exactly $\la_a$ choices of $u$ with $u+\sigma(u)=a$. Hence
\[
|X_a(m)|=\la_a\,|\mathcal{P}(m)|=\la_a p_m.
\]
Now define
\[
\Phi:X_a(m)\longrightarrow \mathcal{P}_a(m)\times G,
\qquad
\Phi(\sigma,u)=((-u)\star\sigma,u).
\]
If $u+\sigma(u)=a$, then
\[
((-u)\star\sigma)(0)=\sigma(u)+u=a,
\]
so $(-u)\star\sigma\in \mathcal{P}_a(m)$. Conversely, given $(\pi,u)\in \mathcal{P}_a(m)\times G$, let $\sigma=u\star\pi$. Then
\[
\sigma(u)=\pi(0)-u=a-u,
\]
so $(\sigma,u)\in X_a(m)$ and $\Phi(\sigma,u)=(\pi,u)$. Thus $\Phi$ is a bijection, and therefore
\[
\la_a p_m=|X_a(m)|=n\,|\mathcal{P}_a(m)|.
\]
This proves the first identity.

For the signed identity, weight every pair $(\sigma,u)\in X_a(m)$ by $\sgn(\sigma)$. Because the action in Lemma~\ref{lem:translation-action} preserves sign, the same bijection gives
\[
\sum_{(\sigma,u)\in X_a(m)}\sgn(\sigma)
=
 n\sum_{\pi\in \mathcal{P}_a(m)}\sgn(\pi).
\]
On the other hand, each $\sigma\in \mathcal{P}(m)$ contributes $\la_a$ times, so the left-hand side equals $\la_a d_m$. The result follows.
\end{proof}

Recall the characters
\begin{equation}\label{eq:hook-chars1}
        \chi^{(n-1,1)}(\sigma)=\Fix(\sigma)-1,
        \qquad
        \chi^{(2,1^{n-2})}(\sigma)=\sgn(\sigma)(\Fix(\sigma)-1).
\end{equation}

\begin{theorem}\label{thm:master-formula}
Let $G$ be a finite abelian group of order $n$, and let
\[
m=\prod_{a\in G}x_a^{\la_a}
\]
be a monomial of degree $n$. Define
\[
r(a):=|\{g\in G:2g=a\}|.
\]
Then
\begin{equation}\label{eq:hook-coefficient}
[m]\imm_{(n-1,1)}(\mathcal M_G)
=
\frac{\sum_{a\in G}r(a)\la_a-n}{n}\,p_m,
\end{equation}
and
\begin{equation}\label{eq:cohook-coefficient}
[m]\imm_{(2,1^{n-2})}(\mathcal M_G)
=
\frac{\sum_{a\in G}r(a)\la_a-n}{n}\,d_m.
\end{equation}
\end{theorem}

\begin{proof}
By (\ref{eq:hook-chars1}),
\[
[m]\imm_{(n-1,1)}(\mathcal M_G)
=
\sum_{\sigma\in \mathcal{P}(m)}(\Fix(\sigma)-1).
\]
Now
\[
\sum_{\sigma\in \mathcal{P}(m)}\Fix(\sigma)
=
\sum_{u\in G}|\{\sigma\in \mathcal{P}(m):\sigma(u)=u\}|.
\]
Fix $u\in G$. The condition $\sigma(u)=u$ is equivalent to $u+\sigma(u)=2u$, and by the same translation argument as in Lemma~\ref{lem:Pa-count}, the number of such permutations equals $|\mathcal{P}_{2u}(m)|$. Hence
\[
\sum_{\sigma\in \mathcal{P}(m)}\Fix(\sigma)
=
\sum_{u\in G}|\mathcal{P}_{2u}(m)|
=
\sum_{a\in G}r(a)|\mathcal{P}_a(m)|.
\]
By Lemma~\ref{lem:Pa-count}, this is
\[
\frac{1}{n}\sum_{a\in G}r(a)\la_a\,p_m.
\]
Subtracting $p_m=|\mathcal{P}(m)|$ gives~\eqref{eq:hook-coefficient}.

For the second identity, by (\ref{eq:hook-chars1}),
\[
[m]\imm_{(2,1^{n-2})}(\mathcal M_G)
=
\sum_{\sigma\in \mathcal{P}(m)}\sgn(\sigma)(\Fix(\sigma)-1).
\]
The same computation as above, but weighted by sign, gives
\[
\sum_{\sigma\in \mathcal{P}(m)}\sgn(\sigma)\Fix(\sigma)
=
\sum_{a\in G}r(a)\sum_{\sigma\in \mathcal{P}_a(m)}\sgn(\sigma)
=
\frac{1}{n}\sum_{a\in G}r(a)\la_a\,d_m,
\]
again by Lemma~\ref{lem:Pa-count}. Subtracting $d_m$ yields~\eqref{eq:cohook-coefficient}.
\end{proof}

\begin{proof}[Proof of Theorem~\ref{near-hookimm}.(1)]
If $n$ is odd, the doubling map on $G$ is bijective, so $r(a)=1$ for every $a\in G$. Since $\sum_{a\in G}\la_a=n$ for every monomial of degree $n$, the scalar factor in~\eqref{eq:hook-coefficient} and~\eqref{eq:cohook-coefficient} is always zero. Hence every coefficient vanishes.
\end{proof}

\begin{proof}[Proof of Theorem~\ref{near-hookimm}.(2)]
Write $n=2m$ with $m$ odd. Since the Sylow $2$-subgroup of $G$ has order $2$, the subgroup $2G$ has index $2$, so $G/2G\cong C_2$. Moreover,
\[
r(a)=
\begin{cases}
2,& a\in 2G,\\
0,& a\notin 2G.
\end{cases}
\]
For a monomial $m=\prod_{a\in G}x_a^{\la_a}$, let
\[
E=\sum_{a\in 2G}\la_a,
\qquad
O=\sum_{a\notin 2G}\la_a.
\]
Then $E+O=n$, and the scalar factor in Theorem~\ref{thm:master-formula} is
\[
\frac{2E-n}{n}=\frac{E-O}{n}.
\]
Assume first that $p_m>0$. By Theorem \ref{Hal}, the multiset corresponding to $m$ is a zero-sum sequence of length $n$ over $G$. Passing to the quotient $G/2G\cong C_2$, we see that the number of terms outside $2G$ must be even; in other words, $O$ is even. But $n/2=m$ is odd, so $E=O=n/2$ is impossible. Hence $E-O\neq 0$, and therefore the factor in~\eqref{eq:hook-coefficient} is nonzero. It follows that
\[
[m]\imm_{(n-1,1)}(\mathcal M_G)\neq 0
\qquad\Longleftrightarrow\qquad
p_m\neq 0.
\]
This proves the first statement.

For the second statement, if $d_m\neq 0$ then certainly $p_m>0$, so the same scalar factor is again nonzero. Therefore~\eqref{eq:cohook-coefficient} implies
\[
[m]\imm_{(2,1^{n-2})}(\mathcal M_G)\neq 0
\qquad\Longleftrightarrow\qquad
d_m\neq 0.
\]
This completes the proof.
\end{proof}

\section{The odd-order immanant twin}\label{sec:twin}

In this section, we consider a general case and prove Theorem \ref{immtwins}.  

For a permutation $\sigma$, let $c_i(\sigma)$ be the number of $i$-cycles of $\sigma$.  It is easy to see that
\begin{equation}\label{eq:h3char}
\chi^{(n-3,1^3)}=\binom{c_1-1}{3}-(c_1-1)c_2+c_3.
\end{equation}
and
\begin{equation}\label{eq:j3char}
\chi^{(n-3,3)}=\binom{c_1}{3}-\binom{c_1}{2}+(c_1-1)c_2+c_3.
\end{equation}
Since $(4,1^{n-4})=(n-3,1^3)'$ and $(2,2,2,1^{n-6})=(n-3,3)'$, subtracting \eqref{eq:j3char} from \eqref{eq:h3char} after sign twist gives
\begin{equation}\label{eq:twin-char-diff}
\chi^{(4,1^{n-4})}(\sigma)-\chi^{(2,2,2,1^{n-6})}(\sigma)
        =\sgn(\sigma)(c_1(\sigma)-1)(1-2c_2(\sigma)).
\end{equation}

Let $A=(a_{ij})$ be an arbitrary $n\times n$ matrix.  Write $A(S|S)$ for the principal matrix obtained by deleting rows and columns indexed by $S$.  Define
\[
        F_1(A)=\sum_i a_{ii}\det A(i|i),
\]
\[
        T_2(A)=\sum_{i<j}a_{ij}a_{ji}\det A(i,j|i,j),
\]
and
\[
        T_{12}(A)=\sum_{\substack{i\notin\{j,k\}\\j<k}}a_{ii}a_{jk}a_{kj}\det A(i,j,k|i,j,k).
\]

\begin{proposition}[principal-minor reduction]\label{prop:twin-reduction}
For every $n\times n$ matrix $A$,
\begin{equation}\label{eq:twin-reduction}
\imm_{(4,1^{n-4})}(A)-\imm_{(2,2,2,1^{n-6})}(A)
        =F_1(A)-\det A+2(T_{12}(A)-T_2(A)).
\end{equation}
\end{proposition}

\begin{proof}
By \eqref{eq:twin-char-diff}, the difference of immanants is
\[
\sum_{\sigma}\sgn(\sigma)(c_1-1)(1-2c_2)\prod_i a_{i,\sigma(i)}.
\]
The determinant sum with one fixed point marked is $F_1(A)$, so the contribution of $c_1-1$ is $F_1(A)-\det A$.  The determinant sum with one transposition marked has an extra negative sign from the transposition.  Thus the contribution of $(c_1-1)c_2$ is $-T_{12}(A)+T_2(A)$.  Multiplying by $-2$ gives \eqref{eq:twin-reduction}.
\end{proof}

We apply this with $A=\mathcal M_G$.  Work over the rational function field $K=\C(x_g:g\in G)$.  The determinant of $\mathcal M_G$ is not the zero polynomial: specializing $x_0=1$ and $x_g=0$ for $g\ne0$ gives the permutation matrix of $a\mapsto -a$.  Hence $\mathcal M_G$ is invertible over $K$.  Its inverse again has group-Hankel form
\[
        \mathcal M_G^{-1}=(y_{a+b})_{a,b\in G}
\]
for rational functions $y_g$.  Equivalently, the functions $x_g$ and $y_g$ satisfy the convolution equations
\begin{equation}\label{eq:conv}
        \sum_{r\in G}x_r y_{r+s}=\delta_{s,0}\qquad(s\in G).
\end{equation}
This follows, for instance, from diagonalizing the group matrix by the characters of $G$.

Let $\Delta=\det \mathcal M_G$.  Jacobi's complementary minor theorem gives
\begin{equation}\label{eq:jacobi1}
        \det \mathcal M_G(i|i)=\Delta y_{2i},
\end{equation}
\begin{equation}\label{eq:jacobi2}
        \det \mathcal M_G(i,j|i,j)=\Delta (y_{2i}y_{2j}-y_{i+j}^2),
\end{equation}
and
\begin{equation}\label{eq:jacobi3}
        \det \mathcal M_G(i,j,k|i,j,k)=\Delta\,\Gamma(i,j,k),
\end{equation}
where
\begin{align*}
\Gamma(i,j,k)={}&y_{2i}y_{2j}y_{2k}+2y_{i+j}y_{i+k}y_{j+k} \\
&-y_{2i}y_{j+k}^2-y_{2j}y_{i+k}^2-y_{2k}y_{i+j}^2.
\end{align*}

\begin{lemma}\label{lem:F1-det}
Let $G$ be a finite abelian group of odd order, then $F_1(\mathcal M_G)=\det \mathcal M_G$.
\end{lemma}

\begin{proof}
Using \eqref{eq:jacobi1},
\[
        F_1(\mathcal M_G)=\Delta\sum_i x_{2i}y_{2i}=\Delta\sum_r x_ry_r=\Delta,
\]
because doubling is a bijection and \eqref{eq:conv} with $s=0$ gives $\sum_rx_ry_r=1$.
\end{proof}

\begin{lemma}\label{lem:T12-T2}
If $G$ has odd order, then $T_{12}(\mathcal M_G)=T_2(\mathcal M_G)$.
\end{lemma}

\begin{proof}
Set
\[
        C=\sum_{s\in G}x_s^2\sum_{t\in G}y_ty_{2s-t},
        \qquad
        S=\sum_{s\in G}x_s^2y_s^2.
\]
By \eqref{eq:jacobi2}, extending the symmetric sum over $i<j$ to an ordered sum gives
\begin{align}\label{eq:T2calc}
\frac{2T_2(\mathcal M_G)}{\Delta}
&=\sum_{i,j}x_{i+j}^2(y_{2i}y_{2j}-y_{i+j}^2)\notag\\
&=C-nS.
\end{align}
Indeed, the map $(i,j)\mapsto (s,t)=(i+j,2i)$ is a bijection of $G^2$, and for fixed $s=i+j$ there are $n$ ordered pairs $(i,j)$.

Similarly, because $\Gamma$ is symmetric and vanishes when two indices coincide,
\begin{equation}\label{eq:T12-start}
        \frac{2T_{12}(\mathcal M_G)}{\Delta}=
        \sum_{i,j,k}x_{2i}x_{j+k}^2\Gamma(i,j,k).
\end{equation}
Split the right-hand side according to the five terms in $\Gamma$ as
\[
        B_1+2B_2-B_3-B_4-B_5.
\]
The first and third terms are
\[
        B_1=C,
        \qquad B_3=nS,
\]
using \eqref{eq:conv} and the same reindexing as above.

For $B_2$, use the bijection
\[
        (i,j,k)\longmapsto (u,v,s)=(i+j,i+k,j+k),
\]
whose inverse exists because $2$ is invertible in $G$.  Then
\[
        B_2=\sum_{u,v,s}x_{u+v-s}x_s^2y_uy_vy_s
        =\sum_{v,s}x_s^2y_vy_s\sum_u x_{u+v-s}y_u.
\]
By \eqref{eq:conv}, the inner sum is $\delta_{v,s}$; hence $B_2=S$.

For $B_4$, use the bijection
\[
        (i,j,k)\longmapsto (t,v,s)=(i-j,2j,j+k).
\]
Then
\[
        B_4=\sum_{t,v,s}x_{v+2t}x_s^2y_vy_{s+t}^2
        =\sum_{s,t}x_s^2y_{s+t}^2\sum_v x_{v+2t}y_v.
\]
The inner sum is $\delta_{2t,0}=\delta_{t,0}$ by \eqref{eq:conv} and oddness.  Therefore $B_4=S$.  By symmetry $B_5=S$.  Substituting into \eqref{eq:T12-start} gives
\[
        \frac{2T_{12}(\mathcal M_G)}{\Delta}=C+2S-nS-S-S=C-nS,
\]
which agrees with \eqref{eq:T2calc}.
\end{proof}

\begin{proof}[Proof of Theorem \ref{immtwins}.]
Combining Proposition \ref{prop:twin-reduction} and Lemmas \ref{lem:F1-det}, \ref{lem:T12-T2}, we obtain the desired result.
\end{proof}

\begin{proposition}\label{prop:twin-even}
For $G=C_n$ with $n\ge6$ even,
\[
[x_0^n]\left(\imm_{(4,1^{n-4})}(\mathcal M_G)-\imm_{(2,2,2,1^{n-6})}(\mathcal M_G)\right)
        =(3-n)(-1)^{(n-2)/2}.
\]
In particular the two immanants are not equal.
\end{proposition}

\begin{proof}
The monomial $x_0^n$ comes only from $\sigma(i)=-i$.  For even $n$ this permutation has two fixed points and $(n-2)/2$ transpositions.  Substituting $c_1=2$ and $c_2=(n-2)/2$ into \eqref{eq:twin-char-diff} gives the desired result.
\end{proof}

\section{Concluding remarks}\label{sec:questions}

In this paper, we proved several  results concerning the immanants $\imm_{\lambda}(\mathcal M_{G})$ of a Cayley table. It is interesting to study the converse of Theorem \ref{PGDG}. It is known that $\mathcal D(G)<\mathcal P(G)$ when $G$ is a non-cyclic abelian group of order 12. We have some partial results (as well as numerical results) strongly suggest that $\mathcal D(G)=\mathcal P(G)$ only if $|G|$ is a prime power. Theorem \ref{immtwins} suggests that it is worthy to study the immanants $\imm_{\lambda}(\mathcal M_{G})$ for general partition $\lambda$ and find more similar relations. Meanwhile, we also have some partial results (as well as numerical results) supporting Conjecture \ref{mainconj}.(3) and it is could be interesting to find other partitions $\lambda$ with $\mathcal I_{\lambda}(G)=\mathcal P(G)$ or $\mathcal D(G)$. These problems are our next targets to study.

\bigskip

\noindent {\bf Acknowledgments.}
    We are grateful to Prof. Arthur L. Yang for helpful suggestions. This work was supported by Guangdong Basic and Applied Basic
	Research Foundation Grant No.2024A1515012564.

\end{document}